\documentclass[]{_interact}
\usepackage{latexsym}
\usepackage{anyfontsize} 
\usepackage{tikz} 
\usetikzlibrary{cd} 
\usepackage{color}
\usepackage{hyperref}
\usepackage{url}
\usepackage{breakurl}
\newcommand{\bburl}[1]{\textcolor{blue}{\url{#1}}}

\makeatletter
\newcommand{\monthyear}[1]{%
  \def\@monthyear{\uppercase{#1}}}
\newcommand{\volnumber}[1]{%
  \def\@volnumber{\uppercase{#1}}}
\makeatother

\theoremstyle{plain}
\numberwithin{equation}{section} 
\newtheorem{thm}{Theorem}[section] 
\newtheorem{theorem}[thm]{Theorem}
\newtheorem{lemma}[thm]{Lemma}

\newtheorem{definition}[thm]{Definition}

\newtheorem{corollary}[thm]{Corollary}

\newtheorem{conjecture}[thm]{Conjecture}

\numberwithin{table}{section} 
\numberwithin{figure}{section}

\begin{document}

\monthyear{Month Year}
\volnumber{Volume, Number}
\setcounter{page}{1}

\title{Iteration Steps of 3x+1 Problem}

\author{
\name{Youchun Luo\textsuperscript{a} \thanks{Email address: math.youchunluo@gmail.com; yc-luo24@mails.tsinghua.edu.cn}}
\affil{\textsuperscript{a}	Tsinghua Shenzhen International Graduate School,
	 						Tsinghua University,
	 						Shenzhen, Guangdong Province, China}
}
\maketitle

{\bf Article type}: research 
\bigskip

\begin{abstract}
	On the 3x+1 problem, given a positive integer $N$, let $D\left( N \right) $, $O\left( N \right) $ and $E\left( N \right) $ denote the total number of iteration steps, the number of odd iteration steps, and the number of even iteration steps, respectively, when $N$ is iterated until it reaches 1. It is straightforward to observe that $D\left( N \right) =O\left( N \right) +E\left( N \right) $. In this paper, we propose a conjecture termed the Weak Residue Conjecture(i.e., $\frac{2^{E\left( N \right)}}{3^{O\left( N \right)}\cdot N}<2$). We prove that if the 3x+1 conjecture is true and the Weak Residue Conjecture is true, there exist nontrivial relationships among $D\left( N \right) $, $O\left( N \right) $, $E\left( N \right) $, i.e., $O\left( N \right) =\lfloor \log _62\cdot D\left( N \right) -\log _6N \rfloor $(this implies that, given $N$, both $O\left( N \right) $ and $E\left( N \right) $ can be directly computed from $D\left( N \right) $), and five more similar equations are derived simultaneously. Similarly, the case of qx+1 problem is studied too.
\end{abstract}


\section{Introduction}

The 3x+1 problem is most simply stated in terms of the Collatz function $C\left( N \right) $ \cite{Lagarias1985,Lagarias2012}:
$$	C(N)=\left\{ \begin{array}{c}
	3N+1,\mathrm{   if }N\equiv 1\left( \mathrm{mod}~2 \right)\\\\
	\frac{N}{2},\mathrm{   if }N\equiv 0\left( \mathrm{mod}~2 \right)\\
	\end{array} \right. .$$

The 3x+1 problem (or Collatz conjecture) is to prove that starting from any positive integer, some iterate of this function takes the value 1. Numerical computations have verified the Collatz conjecture for all $N\le 2.36\times 10^{21}$ \cite{Roosendaal2025}.
For any positive integer $N$, let $D\left( N \right) $, $O\left( N \right) $ and $E\left( N \right) $ denote the total number of iteration steps, the number of odd iteration steps, and the number of even iteration steps, respectively, when $N$ is iterated until it reaches 1. The number of even iteration steps $E\left( N \right) $ is equal to the total stopping time $\sigma _{\infty}\left( N \right) $ \cite{Leavens1992}:
$$E\left( N \right) =\sigma _{\infty}\left( N \right) .$$
The total stopping time $\sigma _{\infty}\left( N \right) $ is the least whole number $k$ such that $T^{\left( k \right)}\left( N \right) =1$, where $T(N)=\frac{3N+1}{2}$ if $N\equiv 1\left( \mathrm{mod}~2 \right) $, $T(N)=\frac{N}{2}$ if $N\equiv 0\left( \mathrm{mod}~2 \right) $. Relevant literature on the total stopping time $\sigma _{\infty}\left( N \right) $ is available in \cite{Andaloro2000,Ladue2018}.
It is evident that there exists a trivial relationship:
$$D\left( N \right) =O\left( N \right) +E\left( N \right) .$$
For example, the iteration of 7:
$$7\rightarrow 22\rightarrow 11\rightarrow 34\rightarrow 17\rightarrow 52\rightarrow 26\rightarrow 13\rightarrow 40\rightarrow 20\rightarrow 10\rightarrow 5\rightarrow 16\rightarrow 8\rightarrow 4\rightarrow 2\rightarrow 1 .$$
Thus, 
$$D\left( 7 \right) =16, O\left( 7 \right) =5, E\left( 7 \right) =11.$$

We establish a nontrivial relationship among $D\left( N \right) $, $O\left( N \right) $, $E\left( N \right) $ by introducing the concept of Residue. According to Eric Roosendaal’s definition \cite{Roosendaal2025}, the Residue of $N$:
$$Res(N)=\frac{2^{E(N)}}{3^{O(N)}\cdot N} .$$
The lower bound $Res\left( N \right) \ge 1$ was previously established in \cite{Andrei2000,Goodwin2015}. For the upper bound of Residue, Roosendaal proposed the strong residue conjecture: $Res(N)\le Res(993)$, where $Res\left( 993 \right) =1.253142144...$. Roosendaal has not identified any number up to $3\times 10^{19}$ with a higher Residue either. These observations suggest that the strong residue conjecture is likely to hold. Moreover, Roosendaal also proposed the Weak Residue Conjecture: A number $Res_{\max}$ exists such that, for every positive integer, $Res\left( N \right) <Res_{\max} $. As a special case of Roosendaal's Weak Residue Conjecture, we propose the eponymous conjecture:
\begin{conjecture}[Weak Residue Conjecture (WRC)] \label{conj:WRC}
	We have
	$$Res(N)<2 .$$
\end{conjecture}

In the third section, we obtain an upper bound of $Res\left( N \right) $: when $O\left( N \right) \ge 20$, it holds that
$$Res\left( N \right) <O\left( N \right) ^{\small{\frac{1}{9}}} .$$
As a corollary, when $O\left( N \right) \le 512$, the \textbf{WRC} is true.

In the fourth section, for the main purpose of this paper, we establish the nontrivial relationship among $D\left( N \right) $, $O\left( N \right) $, $E\left( N \right) $. Assuming both the 3x+1 conjecture and the \textbf{WRC} hold, we prove that,
	$$\begin{aligned}
	O\left( N \right) &=\lfloor \log _62\cdot D\left( N \right) -\log _6N \rfloor ,  \\
	E\left( N \right) &=\lceil \log _63\cdot D\left( N \right) +\log _6N \rceil , \\
	D\left( N \right) &=\lceil \log _26\cdot O\left( N \right) +\log _2N \rceil , \\
	E\left( N \right) &=\lceil \log _23\cdot O\left( N \right) +\log _2N \rceil , \\
	D\left( N \right) &=\lfloor \log _36\cdot E\left( N \right) -\log _3N \rfloor ,  \\
	O\left( N \right) &=\lfloor \log _32\cdot E\left( N \right) -\log _3N \rfloor  .
	\end{aligned}$$
The formulas above depict nontrivial relationships among $D\left( N \right) $, $O\left( N \right) $, $E\left( N \right) $. For example, when $N=7$, we know $D\left( 7 \right) =16$. Using these formulas, we can directly compute $O\left( 7 \right) ,E\left( 7 \right) $:
	$$\begin{aligned}
	O\left( 7 \right) &=\lfloor \log _62\cdot D\left( 7 \right) -\log _67 \rfloor =\lfloor 5.10361 \rfloor =5, \\
	E\left( 7 \right) &=\lceil \log _63\cdot D\left( 7 \right) +\log _67 \rceil =\lceil 10.8964 \rceil =11.
	\end{aligned}$$

In the fifth section, for qx+1 problem, we extend the definition of residue to $res_q\left( N,N\prime \right) $, and derive more general results concerning this function.

	\section{Proof of $Res\left( N \right) \ge 1$}

\begin{definition}[Syracuse map]
	According to Terence Tao's definition \cite{Tao2022}, given a positive integer $N$, the Collatz iteration process from $N$ to 1 generates the Collatz map of $N$. All odd integers (except 1) in the Collatz map are referred to as the Syracuse map of $N$, that is
	$$N_1\rightarrow \cdots \rightarrow N_j\rightarrow \cdots \rightarrow N_i ,$$
	where $N_i\ne 1$, the number of elements in Syracuse map of $N$ equals to the number of odd iteration steps $O\left( N \right) $.
\end{definition}

To facilitate analytical study of the Residue, we introduce an alternative expression for $Res\left( N \right)$.
\begin{lemma}[Product Lemma of Residue]\label{lem:PLR}
	Assuming the 3x+1 conjecture holds, and given a positive integer $N$, the Residue of $N$ can be defined as:	
	$$Res\left( N \right) =\prod_{j=1}^i{\left( 1+\frac{1}{3N_j} \right)} ,$$
	where $N_j$ is the $j$th element of Syracuse map of $N$, and the number of odd iteration steps equal to $i$.
\end{lemma}
In fact, this formula also appears on Roosendaal's web page \cite{Roosendaal2025}.

\begin{proof}	
	\textbf{Case 1}: If $N$ is even, then the Syracuse map of $N$: $\left( N \right) \overset{m_0}{\rightarrow}N_1\rightarrow N_2\rightarrow \cdots \rightarrow N_i\rightarrow \left( 1 \right) $, where $m_0$ denotes the number of even iteration steps from $N$ to $N_1$. Then, the Residue of $N$: $Res\left( N \right) =\frac{2^{E(N)}}{3^{O(N)}\cdot N}$. Since $N_1=\frac{N}{2^{m_0}}$, the Residue of the first odd integer $N_1$:$Res\left( N_1 \right) =\frac{2^{E(N_1)}}{3^{O(N_1)}\cdot N_1}=\frac{2^{E(N)-m_0}}{3^{O(N)}\cdot \frac{N}{2^{m_0}}}=Res\left( N \right) $.
	Thus, it suffices to consider the case where $N$ is odd.
	
	\textbf{Case 2}: If $N$ is odd, then the Syracuse map of $N$: $N=N_1\overset{m_1}{\rightarrow}N_2\rightarrow \cdots \rightarrow N_i\overset{m_i}{\rightarrow}\left( 1 \right) $, where $m_1$ denotes the number of even iteration steps from $N_1$ to $N_2$, $m_i$ denotes the number of even iteration steps from $N_i$ to $1$. The Residue of $N_1$ is: $Res\left( N_1 \right) =\frac{2^{E(N_1)}}{3^{O(N_1)}\cdot N_1}$. Since $N_2=\frac{3N_1+1}{2^{m_1}}$, the Residue of the second odd integer $N_2$ is:
	$$Res\left( N_2 \right) =\frac{2^{E(N_2)}}{3^{O(N_2)}\cdot N_2}=\frac{2^{E(N_1)-m_1}}{3^{O(N_1)-1}\cdot \frac{3N_1+1}{2^{m_1}}}=\frac{Res\left( N_1 \right)}{1+\frac{1}{3N_1}}.$$
	By iterating this relation up to the final odd integer $N_i$, we have
	$$Res\left( N_i \right) =\frac{Res\left( N_1 \right)}{\left( 1+\frac{1}{3N_1} \right) \left( 1+\frac{1}{3N_2} \right) \cdots \left( 1+\frac{1}{3N_{i-1}} \right)}.$$
	Since $1=\frac{3N_i+1}{2^{m_i}}$, it follows that
	$$Res\left( N_i \right) =\frac{2^{m_i}}{3^1\cdot N_i}=\frac{3N_i+1}{3N_i}=1+\frac{1}{3N_i}.$$
	Therefore, the Residue of $N$ satisfies
	\begin{equation*}
		Res\left( N \right) =\prod_{j=1}^i{\left( 1+\frac{1}{3N_j} \right)}.    \unskip\hfill\qedhere
	\end{equation*}
\end{proof}

As a direct consequence of Lemma \ref{lem:PLR}, we have:
\begin{theorem}[Lower Bound of Residue]\label{thm:LBR}
	Assuming the 3x+1 conjecture holds, we have
	$$Res\left( N \right) \ge 1 .$$
\end{theorem}

\begin{proof}
	According to Lemma \ref{lem:PLR}, it always holds that $Res\left( N \right) \ge 1$.
\end{proof}

\section{An upper bound of $Res\left( N \right) $}

To our knowledge, no prior work has established an upper bound for $Res\left( N \right) $. In this section, we present such a bound.

\begin{lemma}\label{lem:mod3}
	Given a positive integer $N$ and its Syracuse map $\left\{ N_j|1\le j\le i \right\} $, for all $j\ge 2$, $N_j$ is not divisible by 3.
\end{lemma}

\begin{proof}
	For any $k$ with $2\le k\le i$ and corresponding $N_k$, suppose $N_k$ is divisible by 3. Write $N_k=3·\left( 2n+1 \right) =6n+3$. Let $N_{k-1}$ be the predecessor of $N_k$ (i.e., $N_{k-1}\rightarrow N_k$). By the 3x+1 rule, there exists $m\in N^*$ such that $N_k=\frac{3N_{k-1}+1}{2^m}$. Rearranging gives
	$$N_{k-1}=\frac{2^mN_k}{3}-\frac{1}{3}=2^m\left( 2n+1 \right) -\frac{1}{3} .$$
	Clearly, $N_{k-1}$ is not an integer, which yields a contradiction.
\end{proof}

\begin{lemma}\label{lem:sum_Nj}
	If $\sum_{j=1}^i{\frac{1}{N_j}}\le a$, then $Res\left( N \right) <e^{\frac{a}{3}}$.
\end{lemma}

\begin{proof}
	By Lemma \ref{lem:PLR}, we have
	$$Res\left( N \right) =\prod_{j=1}^i{\left( 1+\frac{1}{3N_j} \right)} .$$
	Using the logarithmic function,
	$$\ln Res\left( N \right) =\ln \prod_{j=1}^i{\left( 1+\frac{1}{3N_j} \right)}=\sum_{j=1}^i{\ln \left( 1+\frac{1}{3N_j} \right)}<\frac{1}{3}\sum_{j=1}^i{\frac{1}{N_j}} .$$
	Since $\sum_{j=1}^i{\frac{1}{N_j}}\le a$, we have $\ln Res\left( N \right) <\frac{a}{3}$. That is, $Res\left( N \right) <e^{\frac{a}{3}}$.
\end{proof}

\begin{lemma}\label{lem:euler}
	For all $n\in N^*$, we have
	$$\ln n+\gamma <\sum_{k=1}^n{\frac{1}{k}}<\ln n+\gamma +\frac{1}{2n} ,$$
	where $\gamma$ is the Euler constant.
\end{lemma}

\begin{proof}
	Since we have
	$$\sum_{k=1}^n{\frac{1}{k}}=\ln n+\gamma +\frac{1}{2n}+\sum_{k=1}^m{\left( -1 \right) ^k\frac{B_k}{2k\cdot n^{2k}}}+\left( -1 \right) ^{m+1}\theta \frac{B_{m+1}}{\left( 2m+2 \right) \cdot n^{2m+2}} ,$$
	where $\theta \in \left( 0,1 \right) $, $B_k$ denotes the $k$th Bernoulli number. The lemma is thus proved.
\end{proof}

For any $m\in N^*$, let $S_1\left( m \right) =\left\{ k|k\ge 5,k\,\,is\,\,odd,k \not\equiv 0\left( mod\,\,3 \right) ,k\le 6m+3 \right\} $, $S_2\left( m \right) =\left\{ k|k\ge 5,k\,\,is\,\,odd,k \not\equiv 0\left( mod\,\,3 \right) ,k\le 6m+5 \right\} $.
\begin{lemma}\label{lem:harmonic}
	We have
	$$\sum_{k\in S_1\left( m \right)}{\frac{1}{k}}<\ln \left( 3^{\frac{1}{2}}2^{\frac{2}{3}} \right) +\frac{\gamma}{3}-1+\frac{1}{3}\ln m+\frac{5}{6m} .$$
\end{lemma}

\begin{proof}
	For any $m\in N^*$, by Lemma \ref{lem:euler}, we obtain
	$$\sum_{k=1}^{6m+3}{\frac{1}{k}}<\ln\mathrm{(}6m+3)+\gamma +\frac{1}{2(6m+3)},$$
	$$\sum_{2\le k,k\,\,is\,\,even}^{6m+2}{\frac{1}{k}}=\frac{1}{2}\sum_{k=1}^{3m+1}{\frac{1}{k}}>\frac{1}{2}\ln\mathrm{(}3m+1)+\frac{\gamma}{2}.$$
	Thus,
	$$\sum_{1\le k,k\,\,is\,\,odd}^{6m+3}{\frac{1}{k}}=\sum_{k=1}^{6m+3}{\frac{1}{k}}-\sum_{2\le k,k\,\,is\,\,even}^{6m+2}{\frac{1}{k}}<\ln\mathrm{(}\frac{6m+3}{\sqrt{3m+1}})+\frac{\gamma}{2}+\frac{1}{2(6m+3)} .$$
	Similarly, by Lemma \ref{lem:euler}, we obtain that
	$$\sum_{1\le k,k\,\,is\,\,odd}^{2m+1}{\frac{1}{k}}=\sum_{k=1}^{2m+1}{\frac{1}{k}}-\sum_{2\le k,k\,\,is\,\,even}^{2m}{\frac{1}{k}}>\ln\mathrm{(}2m+1)+\gamma -\frac{1}{2}\ln m-\frac{1}{2}\gamma -\frac{1}{4m} .$$
	We now state the following result:
	\allowdisplaybreaks
	\begin{align*}
		\sum_{k\in S_1\left( m \right)}{\frac{1}{k}}&=\sum_{1\le k,k\,\,is\,\,odd}^{6m+3}{\frac{1}{k}}-\frac{1}{3}\sum_{1\le k,k\,\,is\,\,odd}^{2m+1}{\frac{1}{k}}-1 \\
		&<\ln \left( \frac{\left( 6m+3 \right) m^{\small{\frac{1}{6}}}}{\sqrt{3m+1}\left( 2m+1 \right) ^{\small{\frac{1}{3}}}} \right) +\frac{\gamma}{3}-1+\frac{1}{12m}+\frac{1}{2\left( 6m+3 \right)}\\
		&<\ln \left( \frac{3^{\frac{1}{2}}2^{\frac{2}{3}}(m+1)^{\frac{2}{3}}}{m^{\frac{1}{3}}} \right) +\frac{\gamma}{3}-1+\frac{1}{12m}+\frac{1}{2(6m+3)} \\
		&=\ln\mathrm{(}3^{\frac{1}{2}}2^{\frac{2}{3}})+\frac{1}{3}\ln \left( m(\frac{m+1}{m})^2 \right) +\frac{\gamma}{3}-1+\frac{1}{12m}+\frac{1}{2(6m+3)} \\
		&=\ln\mathrm{(}3^{\frac{1}{2}}2^{\frac{2}{3}})+\frac{1}{3}\ln m+\frac{2}{3}\ln\mathrm{(}\frac{m+1}{m})+\frac{\gamma}{3}-1+\frac{1}{12m}+\frac{1}{2(6m+3)} \\
		&<\ln\mathrm{(}3^{\frac{1}{2}}2^{\frac{2}{3}})+\frac{\gamma}{3}-1+\frac{1}{3}\ln m+\frac{2}{3m}+\frac{1}{12m}+\frac{1}{2(6m+3)} \\
		&<\ln\mathrm{(}3^{\frac{1}{2}}2^{\frac{2}{3}})+\frac{\gamma}{3}-1+\frac{1}{3}\ln m+\frac{5}{6m} .  \unskip\hfill\qedhere
	\end{align*}
\end{proof}

\begin{theorem}[Upper Bound Estimate of Residue] \label{thm:UBR}
	Assuming the 3x+1 conjecture holds, if $O\left( N \right) \ge 20$, then
	$$Res\left( N \right) <O\left( N \right) ^{\frac{1}{9}} .$$
\end{theorem}

\begin{proof}
	Consider the sum $\sum_{k\in S_1\left( m \right)}{\frac{1}{k}}$ and $\sum_{k\in S_2\left( m \right)}{\frac{1}{k}}$. 
	The sum $\sum_{k\in S_1\left( m \right)}{\frac{1}{k}}$ includes the smallest $2m$ odd integers that are not divisible by 3. Similarly, $\sum_{k\in S_2\left( m \right)}{\frac{1}{k}}$ includes the smallest $2m+1$ odd integers not divisible by 3.
	
	\textbf{Case 1}: There are $2m$ terms in $\sum_{k\in S_1\left( m \right)}{\frac{1}{k}}$, so $O\left( N \right) =i=2m$. By Lemma \ref{lem:mod3}, two cases arise:
	
	(1) All $N_j$ are not divisible by 3. Since $\sum_{k\in S_1\left( m \right)}{\frac{1}{k}}$ includes the smallest $2m$ odd integers, we obtain
	$$\sum_{j=1}^{2m}{\frac{1}{N_j}}\le \sum_{k\in S_1\left( m \right)}{\frac{1}{k}} .$$
	(2) $N_1\equiv 0\left( mod\,\,3 \right) $, then its possible values include 3,9,15,21,27,33,39,.... If $N_1\ge 39$, then we obtain
	$$\sum_{j=1}^{2m}{\frac{1}{N_j}}\le \sum_{k\in S_1\left( m \right)}{\frac{1}{k}}+\frac{1}{39} .$$
	By Lemma \ref{lem:harmonic}, if $N_1\notin \left\{ 3,9,15,21,27,33 \right\} $, then
	$$\sum_{j=1}^{2m}{\frac{1}{N_j}}\le \sum_{k\in S_1\left( m \right)}{\frac{1}{k}}+\frac{1}{39}<\ln\mathrm{(}3^{\frac{1}{2}}2^{\frac{2}{3}})+\frac{\gamma}{3}-\frac{38}{39}+\frac{1}{3}\ln m+\frac{5}{6m} .$$
	By Lemma \ref{lem:sum_Nj}, we obtain
	\begin{align*}
		Res\left( N \right) &<e^{\ln\mathrm{(}3^{\frac{1}{6}}2^{\frac{2}{9}})+\frac{\gamma}{9}-\frac{38}{117}+\frac{1}{9}\ln m+\frac{5}{18m}} \\
		&=3^{\frac{1}{6}}2^{\frac{2}{9}}e^{\frac{\gamma}{9}-\frac{38}{117}}m^{\frac{1}{9}}e^{\frac{5}{18m}} \\
		&=\underset{0.999467... }{\underbrace{3^{\frac{1}{6}}2^{\frac{1}{9}}e^{\frac{\gamma}{9}-\frac{38}{117}}}}e^{\frac{5}{9O(N)}}\,\,O(N)^{\frac{1}{9}} .
	\end{align*}
	If $\,\,O(N)\ge 1043$, then
	$$0.999467<3^{\frac{1}{6}}2^{\frac{1}{9}}e^{\frac{\gamma}{9}-\frac{38}{117}}e^{\frac{5}{9O(N)}}<0.99999981 .$$
	Therefore, if $\,\,O(N)\ge 1043$ and $O\left( N \right) $ is even, $N_1\notin \left\{ 3,9,15,21,27,33 \right\} $, then we obtain
	$$Res\left( N \right) <O\left( N \right) ^{\frac{1}{9}} .$$
	
	\textbf{Case 2}: By Lemma \ref{lem:harmonic}, we obtain
	\begin{align*}
		\sum_{k\in S_2\left( m \right)}{\frac{1}{k}}&<\ln\mathrm{(}3^{\frac{1}{2}}2^{\frac{2}{3}})+\frac{\gamma}{3}-1+\frac{1}{3}\ln m+\frac{5}{6m}+\frac{1}{6m+5} \\
		&<\ln\mathrm{(}3^{\frac{1}{2}}2^{\frac{2}{3}})+\frac{\gamma}{3}-1+\frac{1}{3}\ln m+\frac{1}{m} .
	\end{align*}
	There are $2m+1$ terms in $\sum_{k\in S_2\left( m \right)}{\frac{1}{k}}$, so $O\left( N \right) =i=2m+1$. By Lemma \ref{lem:mod3}, two cases arise:
	
	(1) All $N_j$ are not divisible by 3. Since $\sum_{k\in S_2\left( m \right)}{\frac{1}{k}}$ includes the smallest $2m+1$ odd integers not divisible by 3, we obtain
	$$\sum_{j=1}^{2m+1}{\frac{1}{N_j}}\le \sum_{k\in S_2\left( m \right)}{\frac{1}{k}} .$$
	(2) $N_1\equiv 0\left( mod\,\,3 \right) $, then its possible values include 3,9,15,21,27,33,39,.... If $N_1\ge 39$, then we obtain
	$$\sum_{j=1}^{2m+1}{\frac{1}{N_j}}\le \sum_{k\in S_2\left( m \right)}{\frac{1}{k}}+\frac{1}{39} .$$
	Therefore, if $N_1\notin \left\{ 3,9,15,21,27,33 \right\} $, then
	$$\sum_{j=1}^{2m+1}{\frac{1}{N_j}\le \sum_{k\in S_2\left( m \right)}{\frac{1}{k}}}+\frac{1}{39}<\ln\mathrm{(}3^{\frac{1}{2}}2^{\frac{2}{3}})+\frac{\gamma}{3}-\frac{38}{39}+\frac{1}{3}\ln m+\frac{1}{m} .$$
	
	By Lemma \ref{lem:sum_Nj}, we obtain
	\begin{align*}
		Res(N)&<e^{\ln\mathrm{(}3^{\frac{1}{6}}2^{\frac{2}{9}})+\frac{\gamma}{9}-\frac{38}{117}+\frac{1}{9}\ln m+\frac{1}{3m}} \\
		&=3^{\frac{1}{6}}2^{\frac{2}{9}}e^{\frac{\gamma}{9}-\frac{38}{117}}m^{\frac{1}{9}}e^{\frac{1}{3m}} \\
		&=3^{\frac{1}{6}}2^{\frac{1}{9}}e^{\frac{\gamma}{9}-\frac{38}{117}}e^{\frac{2}{3(O(N)-1)}}\left( O(N)-1 \right) ^{\frac{1}{9}} \\
		&<\underset{0.999467... }{\underbrace{3^{\frac{1}{6}}2^{\frac{1}{9}}e^{\frac{\gamma}{9}-\frac{38}{117}}}}e^{\frac{2}{3(O(N)-1)}}O(N)^{\frac{1}{9}} .
	\end{align*}
	If $O\left( N \right) \ge 1253$, then
	$$0.999467<3^{\frac{1}{6}}2^{\frac{1}{9}}e^{\frac{\gamma}{9}-\frac{38}{117}}e^{\frac{2}{3(O(N)-1)}}<0.99999964 .$$
	Therefore, if $O\left( N \right) \ge 1253$ and $O\left( N \right) $ is odd, and $N_1\notin \left\{ 3,9,15,21,27,33 \right\} $, then we obtain
	$$Res\left( N \right) <O\left( N \right) ^{\frac{1}{9}} .$$
	Combining Case 1 and Case 2, we conclude that if $O\left( N \right) \ge 1253$, $N_1\notin \left\{ 3,9,15,21,27,33 \right\} $, then
	$$Res\left( N \right) <O\left( N \right) ^{\frac{1}{9}} .$$
	
	Via simple computer programming, we can verify that when $20\le O\left( N \right) \le 1252$, $N_1\notin \left\{ 3,9,15,21,27,33 \right\} $, the result holds
	$$\left( 1+\frac{1}{3\times 39} \right) \cdot 
	\prod_{
		\scriptstyle 5\le j,j\,\,is\,\,odd \atop \scriptstyle j \not\equiv 0\left( mod\,\,3 \right)
	}^{\scriptstyle O(N)~terms}{\left( 1+\frac{1}{3j}\right)}<O\left( N \right) ^{\frac{1}{9}} .$$
	It follows that
	$$Res\left( N \right) =\prod_{j=1}^{O(N)}{\left( 1+\frac{1}{3N_j} \right)}\le \left( 1+\frac{1}{3\times 39} \right) \cdot \prod_{
		\scriptstyle 5\le j,j\,\,is\,\,odd \atop \scriptstyle j \not\equiv 0\left( mod\,\,3 \right)
	}^{\scriptstyle O(N)~terms}{\left( 1+\frac{1}{3j}\right)} .$$
	Therefore, if $O\left( N \right) \ge 20$, $N_1\notin \left\{ 3,9,15,21,27,33 \right\} $, the result holds:
	$$Res\left( N \right) <O\left( N \right) ^{\frac{1}{9}} .$$
	
	Now consider the case $O\left( N \right) \ge 20$, $N_1\in \left\{ 3,9,15,21,27,33 \right\} $. When $N_1=3,9,15,21,27,33$, for these values of $N_1$, the corresponding $O\left( N \right) =2,6,5,1,41,8$. Therefore, only $N_1=27$ satisfies both $O\left( N \right) \ge 20$ and $N_1\in \left\{ 3,9,15,21,27,33 \right\} $. In this case, when $N_1=27$,
	$$Res\left( 27 \right) =1.1988... <O\left( 27 \right) ^{\frac{1}{9}}=1.5107... .$$
	Finally, we conclude that if $O\left( N \right) \ge 20$, the result holds:
	\begin{equation*}
		Res\left( N \right) <O\left( N \right) ^{\frac{1}{9}} .   \unskip\hfill\qedhere
	\end{equation*}
\end{proof}

Although a complete proof of the \textbf{WRC} remains elusive, we can still establish the following result:
\begin{corollary}
	Assuming the 3x+1 conjecture holds, if $O\left( N \right) \le 512$, then the \textbf{WRC} is valid.
\end{corollary}

\begin{proof}
	By Theorem \ref{thm:UBR}, when $20\le O\left( N \right) \le 512$, we obtain
	$$Res\left( N \right) <O\left( N \right) ^{\frac{1}{9}}\le 512^{\frac{1}{9}}=2 .$$
	For any $O\left( N \right) \le 19$, there are two possible cases:
	
	(1) All $N_j$ are not divisible by 3. Then we have
	$$\sum_{j=1}^{O(N)}{\frac{1}{N_j}}\le \sum_{
		\scriptstyle 5\le k,k\,\,is\,\,odd \atop \scriptstyle k \not\equiv 0\left( mod\,\,3 \right)
	}^{\scriptstyle O(N)~terms}{\frac{1}{k}}\le \sum_{
		\scriptstyle 5\le k,k\,\,is\,\,odd \atop \scriptstyle k \not\equiv 0\left( mod\,\,3 \right)
	}^{19~terms}{\frac{1}{k}} .$$
	
	(2) $N_1\equiv 0\left( mod\,\,3 \right) $, then we have
	$$\sum_{j=1}^{O(N)}{\frac{1}{N_j}}\le \sum_{
		\scriptstyle 5\le k,k\,\,is\,\,odd \atop \scriptstyle k \not\equiv 0\left( mod\,\,3 \right)
	}^{\scriptstyle O(N)~terms}{\frac{1}{k}}+\frac{1}{3}\le \sum_{
		\scriptstyle 5\le k,k\,\,is\,\,odd \atop \scriptstyle k \not\equiv 0\left( mod\,\,3 \right)
	}^{19~terms}{\frac{1}{k}}+\frac{1}{3} .$$
	Combining these two cases, we obtain
	\begin{align*}
	\sum_{j=1}^{O(N)}{\frac{1}{N_j}}&\le \sum_{
		\scriptstyle 5\le k,k\,\,is\,\,odd \atop \scriptstyle k \not\equiv 0\left( mod\,\,3 \right)
		}^{19~terms}{\frac{1}{k}}+\frac{1}{3}
		\\
		&=\left( \begin{array}{c}
			\frac{1}{5}+\frac{1}{7}+\frac{1}{11}+\frac{1}{13}+\frac{1}{17}+\frac{1}{19}+\frac{1}{23}+\frac{1}{25}+\frac{1}{29}+\frac{1}{31}\\
			+\frac{1}{35}+\frac{1}{37}+\frac{1}{41}+\frac{1}{43}+\frac{1}{47}+\frac{1}{49}+\frac{1}{53}+\frac{1}{55}+\frac{1}{59}\\
		\end{array} \right) +\frac{1}{3}
		\\
		&=1.3046...  .
	\end{align*}
	By Lemma \ref{lem:sum_Nj}, we obtain
	\begin{equation*}
		Res\left( N \right) <e^{\frac{1.3047}{3}}<1.55 .   \unskip\hfill\qedhere
	\end{equation*}
\end{proof}

\section{Main Results}	
In this section, we explore the remarkable relationships among $D\left( N \right)$, $O\left( N \right)$, $E\left( N \right)$. Typically, $D\left( N \right)$, $O\left( N \right)$, $E\left( N \right)$ vary irregularly as $N$ increases. By analyzing the properties of the Residue, we derive six explicit equations.
\begin{theorem}[DOE Theorem]
	Given a positive integer $N$, assume the 3x+1 conjecture holds and the \textbf{WRC} is valid. Then, knowing any one of $D\left( N \right) , O\left( N \right) , E\left( N \right) $, the other two can be computed using the following formulas:
	\begin{align}
		O\left( N \right) &=\lfloor \log _62\cdot D\left( N \right) -\log _6N \rfloor  \label{eq1} ,\\
		E\left( N \right) &=\lceil \log _63\cdot D\left( N \right) +\log _6N \rceil \label{eq2} ,\\
		D\left( N \right) &=\lceil \log _26\cdot O\left( N \right) +\log _2N \rceil \label{eq3} ,\\
		E\left( N \right) &=\lceil \log _23\cdot O\left( N \right) +\log _2N \rceil \label{eq4} ,\\
		D\left( N \right) &=\lfloor \log _36\cdot E\left( N \right) -\log _3N \rfloor  \label{eq5} ,\\
		O\left( N \right) &=\lfloor \log _32\cdot E\left( N \right) -\log _3N \rfloor  \label{eq6} .
	\end{align}
\end{theorem}
In fact, \eqref{eq3} is a conjecture proposed by Rafael Ruggiero (2019) \cite{Ruggiero2019}.

\begin{proof}
	By Theorem \ref{thm:LBR} and \textbf{WRC}, we always have $1\le Res\left( N \right) <2$. Applying the logarithmic function, we obtain $0\le \log _2\frac{2^{E(N)}}{3^{O(N)}\cdot N}<1$. That is,
	\begin{equation}
		0\le E(N)-\log _23\cdot O(N)-\log _2N<1. \label{eq7}
	\end{equation}
	Replacing $E$ with $D-O$ and $O$ with $D-E$, respectively, we obtain
	\begin{align}
		0&\le D(N)-\log _26\cdot O(N)-\log _2N<1, \label{eq8} \\
		0&\le \log _26\cdot E(N)-\log _23\cdot D(N)-\log _2N<1. \label{eq9}
	\end{align}
	By transforming \eqref{eq7}, \eqref{eq8} and \eqref{eq9}, respectively, six inequalities are derived,
	\begin{align}
		\log _23\cdot O\left( N \right) +\log _2N&\le E(N)<\log _23\cdot O\left( N \right) +\log _2N+1, \label{eq10}
		\\
		\frac{E(N)-\log _2N}{\log _23}-\frac{1}{\log _23}&<O(N)\le \frac{E(N)-\log _2N}{\log _23}, \label{eq11}
		\\
		\log _26\cdot O\left( N \right) +\log _2N&\le D(N)<\log _26\cdot O\left( N \right) +\log _2N+1, \label{eq12}
		\\
		\frac{D(N)-\log _2N}{\log _26}-\frac{1}{\log _26}&<O(N)\le \frac{D(N)-\log _2N}{\log _26}, \label{eq13}
		\\
		\frac{\log _23\cdot D(N)+\log _2N}{\log _26}&\le E(N)<\frac{\log _23\cdot D(N)+\log _2N}{\log _26}+\frac{1}{\log _26}, \label{eq14}
		\\
		\frac{\log _26\cdot E(N)-\log _2N}{\log _23}-\frac{1}{\log _23}&<D(N)\le \frac{\log _26\cdot E(N)-\log _2N}{\log _23}. \label{eq15}
	\end{align}
	Clearly, each of these intervals has length less than 1, and contains exactly one integer. Therefore, \eqref{eq10} implies \eqref{eq4}, \eqref{eq11} implies \eqref{eq6}, \eqref{eq12} implies \eqref{eq3}, \eqref{eq13} implies \eqref{eq1}, \eqref{eq14} implies \eqref{eq2}, and \eqref{eq15} implies \eqref{eq5}.
\end{proof}

	\section{Extension of $Res\left( N \right) $}
For qx+1 problem, where $q$ is an odd integer with $q\ge 3$, given an integer $N$, the q-Collatz function is defined as follows:
$$C_q(N)=\left\{ \begin{array}{c}
	qN+1,\mathrm{   if }N\equiv 1\left( \mathrm{mod}~2 \right)\\
	\\
	\frac{N}{2},\mathrm{   if }N\equiv 0\left( \mathrm{mod}~2 \right)\\
\end{array} \right. .$$

\begin{definition}
	If $N$ eventually iterates to 1, define the q-Residue as follows:
	$$Res_q\left( N \right) =\frac{2^{E\left( N \right)}}{q^{O\left( N \right)}\cdot N} .$$
\end{definition}

However, when $q>3$, the density of integers that can iterate to $1$ in the set of all positive integers is believed to equal to $0$ \cite{Kontorovich2009}. Thus, $Res_q\left( N \right) $ is limited in scope. It is therefore necessary to define residue for all positive integers, regardless of whether $N$ reaches 1. Moreover, we also consider the case where $N$ is a negative integer.

\begin{definition}
	For qx+1 problem, given a positive or negative integer $N$, let $N\prime $ be the result after $e$ even steps and $o$ odd steps, assuming all iteration values are non-repeating.	Define
	$$res_q\left( N,N\prime \right) =\frac{2^e}{q^o\cdot \small{\frac{N}{N\prime}}} .$$
\end{definition}
In particular, we have $Res_q\left( N \right) =res_q\left( N,1 \right) $.

\begin{lemma}[Product Lemma of residue]\label{lem:PLr}
	Given a positive or negative integer $N$, the following identity holds:
	$$res_q\left( N,N\prime \right) =\prod_{j=1}^o{\left( 1+\frac{1}{q\cdot N_j} \right)} ,$$
	where $N_j$ is the $j$th odd integer encountered during the iteration from $N$ to $N\prime $(except $N\prime $), and the number of odd iteration steps is $o$.
\end{lemma}

\begin{proof}
	\textbf{Case 1}: If $N$ is even, then the Syracuse map of $N$: $\left( N \right) \overset{m_0}{\rightarrow}N_1\rightarrow N_2\rightarrow \cdots \rightarrow N_o\rightarrow \left( N\prime \right) $, where $m_0$ denotes the number of even iteration steps from $N$ to $N_1$. Then, $res_q\left( N,N\prime \right) =\frac{2^e}{q^o\cdot \frac{N}{N\prime}}$. Since $N_1=\frac{N}{2^{m_0}}$, we obtain $res_q\left( N_1,N\prime \right) =\frac{2^{e-m_0}}{q^o\cdot \frac{N_1}{N\prime}}=res_q\left( N,N\prime \right) $.
	Therefore, it suffices to consider the case where $N$ is odd.
	
	\textbf{Case 2}: If $N$ is odd, then the Syracuse map of $N$: $N=N_1\overset{m_1}{\rightarrow}N_2\rightarrow \cdots \rightarrow N_o\overset{m_o}{\rightarrow}\left( N\prime \right) $, where $m_1$ denotes the number of even iteration steps from $N_1$ to $N_2$, $m_o$ denotes the number of even iteration steps from $N_o$ to $N\prime $. We have $res_q\left( N_1,N\prime \right) =\frac{2^e}{q^o\cdot \frac{N_1}{N\prime}}$. Since $N_2=\frac{qN_1+1}{2^{m_1}}$, we obtain
	$$res_q\left( N_2,N\prime \right) =\frac{2^{e-m_1}}{q^{o-1}\cdot \frac{N_2}{N\prime}}=\frac{2^e}{q^{o-1}\cdot \frac{qN_1+1}{N\prime}}=\frac{res_q\left( N_1,N\prime \right)}{1+\frac{1}{qN_1}} .$$
	By iterating this relation up to the final odd integer $N_o$, we obtain
	$$res_q\left( N_o,N\prime \right) =\frac{res_q\left( N_1,N\prime \right)}{\left( 1+\frac{1}{qN_1} \right) \left( 1+\frac{1}{qN_2} \right) \cdots \left( 1+\frac{1}{qN_{o-1}} \right)} .$$
	Since $N\prime =\frac{qN_o+1}{2^{m_o}}\left( m_o\ge 0 \right) $, it follows that
	$$res_q\left( N_o,N\prime \right) =\frac{2^{m_o}}{q^1\cdot \frac{N_o}{N\prime}}=1+\frac{1}{qN_o} .$$
	Finally, we obtain the following result:
	\begin{equation*}
		res_q\left( N,N\prime \right) =\prod_{j=1}^o{\left( 1+\frac{1}{q\cdot N_j} \right)} .  \unskip\hfill\qedhere
	\end{equation*}
\end{proof}

As a direct consequence of Lemma \ref{lem:PLr}, we have:
\begin{theorem}[Lower Bound of residue for Positive Integer] \label{thm:LBrPI}
	Given a positive integer $N$, we have
	$$res_q\left( N,N\prime \right) \ge 1 .$$
\end{theorem}
\begin{proof}
	By Lemma \ref{lem:PLr}, we always have $res_q\left( N,N\prime \right) \ge 1$.
\end{proof}

\begin{theorem}[Upper Bound of residue for Negative Integer] \label{thm:LBrNI}
	Given a negative integer $N$, we have
	$$res_q\left( N,N\prime \right) \le 1. $$
\end{theorem}
\begin{proof}
	By Lemma \ref{lem:PLr}, since all $N_j$ are negative odd integers, we always have $res_q\left( N,N\prime \right) \le 1$.
\end{proof}

An interesting topic is whether $res_q\left( N,N\prime \right) $ is bounded. Based on extensive numerical calculations, we propose the following:
\begin{conjecture}
	For any odd integer $q\ge 3$:
	
	{\normalfont(a)} If $N$ is positive, then $\exists R_q<+\infty $ such that $res_q\left( N,N\prime \right) \le R_q $ for all $N,N\prime $;
	
	{\normalfont(b)} If $N$ is negative, then $\exists r_q>0$ such that $r_q\le res_q\left( N,N\prime \right) $ for all $N,N\prime $.
\end{conjecture}

\begin{conjecture}[General Strong Residue Conjecture (GSRC)]
	We conjecture that:
	
	{\normalfont(a)} $R_q<1+\frac{2}{q} $;
	
	{\normalfont(b)} $1-\frac{2}{q}<r_q$.
\end{conjecture}

\begin{conjecture}[General Weak Residue Conjecture (GWRC)] \label{conj:GWRC}
	We conjecture that:
	
	{\normalfont(a)} $R_q<2 $;
	
	{\normalfont(b)} $\frac{1}{2}<r_q$.
\end{conjecture}

Estimated values of the first few $R_q$: $R_3\approx 1.33,R_5\approx 1.28,R_7\approx 1.14,R_9\approx 1.21$. 
Estimated values of the first few $r_q$: $r_3\approx 0.66,r_5\approx 0.8,r_7\approx 0.77,r_9\approx 0.88$.

The lower bound of $R_q$ is derived as follows.
\begin{theorem}[Lower Bound of $R_q$]
	We have
	$$R_q\ge 1+\frac{1}{q}.$$
\end{theorem}

\begin{proof}
	Let $N=1,N\prime =q+1$, the iteration process: $1\rightarrow q+1$ contains no repetition. By Lemma \ref{lem:PLr}, we obtain $R_q\ge res_q\left( 1,q+1 \right) =1+\frac{1}{q\cdot 1}=1+\frac{1}{q}$.
\end{proof}

The upper bound of $r_q$ is derived as follows.
\begin{theorem}[Upper Bound of $r_q$]
	We have
	$$r_q\le 1-\frac{1}{q}. $$
\end{theorem}

\begin{proof}
	Let $N=-1,N\prime =-q+1$, the iteration process: $-1\rightarrow -q+1$ contains no repetition. By Lemma \ref{lem:PLr}, we obtain $r_q\le res_q\left( -1,-q+1 \right) =1+\frac{1}{q\cdot \left( -1 \right)}=1-\frac{1}{q}$.
\end{proof}

\begin{theorem}[doe Theorem for Positive Integers]
	For qx+1 problem, given a positive integer $N$, let $d\left( N \right) $, $o\left( N \right) $ and $e\left( N \right) $ denote the total number of iteration steps, the number of odd iteration steps, and the number of even iteration steps, respectively, in a non-repeating process from $N$ to $N\prime $. If \textbf{GWRC.a} holds, then we have
	\begin{align}
		o&=\lfloor \log _{2q}2\cdot d-\log _{2q}\left( \small{N/N\prime} \right) \rfloor , \label{eq_positive_1} \\
		e&=\lceil \log _{2q}q\cdot d+\log _{2q}\left( \small{N/N\prime} \right) \rceil , \label{eq_positive_2} \\
		d&=\lceil \log _22q\cdot o+\log _2\left( \small{N/N\prime} \right) \rceil , \label{eq_positive_3} \\
		e&=\lceil \log _2q\cdot o+\log _2\left( \small{N/N\prime} \right) \rceil , \label{eq_positive_4} \\
		d&=\lfloor \log _q2q\cdot e-\log _q\left( \small{N/N\prime} \right) \rfloor , \label{eq_positive_5} \\
		o&=\lfloor \log _q2\cdot e-\log _q\left( \small{N/N\prime} \right) \rfloor . \label{eq_positive_6} 
	\end{align}
\end{theorem}

\begin{proof}
	By Theorem \ref{thm:LBrPI} and \textbf{GWRC.a}, we always have $1\le res_q\left( N,N\prime \right) <2$. Applying the logarithmic function, we obtain $0\le \log _2\frac{2^e}{q^o\cdot \small{\frac{N}{N\prime}}}<1$. That is,
	\begin{equation}
		0\le e-\log _2q\cdot o-\log _2\left( N/N\prime \right) <1. \label{eq_positive_7}
	\end{equation}
	Replacing $e$ with $d-o$ and $o$ with $d-e$, respectively, we obtain
	\begin{align}
		0&\le d-\log _22q\cdot o-\log _2\left( N/N\prime \right) <1, \label{eq_positive_8}
		\\
		0&\le \log _22q\cdot e-\log _2q\cdot d-\log _2\left( N/N\prime \right) <1. \label{eq_positive_9}
	\end{align}
	By transforming \eqref{eq_positive_7}, \eqref{eq_positive_8} and \eqref{eq_positive_9}, respectively, six inequalities are derived,
	\begin{align}
		\log _2q\cdot o+\log _2\left( N/N\prime \right) &\le e<\log _2q\cdot o+\log _2\left( N/N\prime \right) +1, \label{eq_positive_10}
		\\
		\frac{e-\log _2\left( N/N\prime \right)}{\log _2q}-\frac{1}{\log _2q}&<o\le \frac{e-\log _2\left( N/N\prime \right)}{\log _2q}, \label{eq_positive_11}
		\\
		\log _22q\cdot o+\log _2\left( N/N\prime \right) &\le d<\log _22q\cdot o+\log _2\left( N/N\prime \right) +1, \label{eq_positive_12}
		\\
		\frac{d-\log _2\left( N/N\prime \right)}{\log _22q}-\frac{1}{\log _22q}&<o\le \frac{d-\log _2\left( N/N\prime \right)}{\log _22q}, \label{eq_positive_13}
		\\
		\frac{\log _2q\cdot d+\log _2\left( N/N\prime \right)}{\log _22q}&\le e<\frac{\log _2q\cdot d+\log _2\left( N/N\prime \right)}{\log _22q}+\frac{1}{\log _22q}, \label{eq_positive_14}
		\\
		\frac{\log _22q\cdot e-\log _2\left( N/N\prime \right)}{\log _2q}-\frac{1}{\log _2q}&<d\le \frac{\log _22q\cdot e-\log _2\left( N/N\prime \right)}{\log _2q}. \label{eq_positive_15}
	\end{align}
	Clearly, each of these intervals has length less than 1 and contains exactly one integer. Therefore, \eqref{eq_positive_10} implies \eqref{eq_positive_4}, \eqref{eq_positive_11} implies \eqref{eq_positive_6}, \eqref{eq_positive_12} implies \eqref{eq_positive_3}, \eqref{eq_positive_13} implies \eqref{eq_positive_1}, \eqref{eq_positive_14} implies \eqref{eq_positive_2}, and \eqref{eq_positive_15} implies \eqref{eq_positive_5}.
\end{proof}

\begin{theorem}[doe Theorem for Negative Integers]
	For qx+1 problem, given a negative integer $N$, let $d\left( N \right) $, $o\left( N \right) $ and $e\left( N \right) $ denote the total number of iteration steps, the number of odd iteration steps, and the number of even iteration steps, respectively, in a non-repeating process from $N$ to $N\prime $. If \textbf{GWRC.b} holds, then we have
	\begin{align}	
		o&=\lceil \log _{2q}2\cdot d-\log _{2q}\left( N/N\prime \right) \rceil , \label{eq_negative_1} \\
		e&=\lfloor \log _{2q}q\cdot d+\log _{2q}\left( N/N\prime \right) \rfloor , \label{eq_negative_2} \\
		d&=\lfloor \log _22q\cdot o+\log _2\left( N/N\prime \right) \rfloor , \label{eq_negative_3} \\
		e&=\lfloor \log _2q\cdot o+\log _2\left( N/N\prime \right) \rfloor , \label{eq_negative_4} \\
		d&=\lceil \log _q2q\cdot e-\log _q\left( N/N\prime \right) \rceil , \label{eq_negative_5} \\
		o&=\lceil \log _q2\cdot e-\log _q\left( N/N\prime \right) \rceil . \label{eq_negative_6} 
	\end{align}
\end{theorem}

\begin{proof}
	By Theorem \ref{thm:LBrNI} and \textbf{GWRC.b}, we always have $\frac{1}{2}<res_q\left( N,N\prime \right) \le 1$. Applying the logarithmic function, we obtain $-1<\log _2\frac{2^e}{q^o\cdot \small{\frac{N}{N\prime}}}\le 0$. That is,
	\begin{equation}
		-1<e-\log _2q\cdot o-\log _2\left( N/N\prime \right) \le 0. \label{eq_negative_7}
	\end{equation}
	Replacing $e$ with $d-o$ and $o$ with $d-e$, respectively, we obtain
	\begin{align}
		-1&<d-\log _22q\cdot o-\log _2\left( N/N\prime \right) \le 0, \label{eq_negative_8}
		\\
		-1&<\log _22q\cdot e-\log _2q\cdot d-\log _2\left( N/N\prime \right) \le 0. \label{eq_negative_9}
	\end{align}
	By transforming \eqref{eq_negative_7}, \eqref{eq_negative_8} and \eqref{eq_negative_9}, respectively, six inequalities are derived,
	\begin{align}
		\log _2q\cdot o+\log _2\left( N/N\prime \right) -1&<e\le \log _2q\cdot o+\log _2\left( N/N\prime \right) , \label{eq_negative_10}
		\\
		\frac{e-\log _2\left( N/N\prime \right)}{\log _2q}&\le o<\frac{e-\log _2\left( N/N\prime \right)}{\log _2q}+\frac{1}{\log _2q}, \label{eq_negative_11}
		\\
		\log _22q\cdot o+\log _2\left( N/N\prime \right) -1&<d\le \log _22q\cdot o+\log _2\left( N/N\prime \right) , \label{eq_negative_12}
		\\
		\frac{d-\log _2\left( N/N\prime \right)}{\log _22q}&\le o<\frac{d-\log _2\left( N/N\prime \right)}{\log _22q}+\frac{1}{\log _22q}, \label{eq_negative_13}
		\\
		\frac{\log _2q\cdot d+\log _2\left( N/N\prime \right)}{\log _22q}-\frac{1}{\log _22q}&<e\le \frac{\log _2q\cdot d+\log _2\left( N/N\prime \right)}{\log _22q}, \label{eq_negative_14}
		\\
		\frac{\log _22q\cdot e-\log _2\left( N/N\prime \right)}{\log _2q}&\le d<\frac{\log _22q\cdot e-\log _2\left( N/N\prime \right)}{\log _2q}+\frac{1}{\log _2q}. \label{eq_negative_15}		
	\end{align}
	Each of these intervals has length less than 1 and therefore contains exactly one integer. Therefore, \eqref{eq_negative_10} implies \eqref{eq_negative_4}, \eqref{eq_negative_11} implies \eqref{eq_negative_6}, \eqref{eq_negative_12} implies \eqref{eq_negative_3}, \eqref{eq_negative_13} implies \eqref{eq_negative_1}, \eqref{eq_negative_14} implies \eqref{eq_negative_2}, and \eqref{eq_negative_15} implies \eqref{eq_negative_5}.
\end{proof}


\section*{Acknowledgements}
We thank Nan Ou, Zihan Xiang, Eric Roosendaal, Mohamed Lhachimi, and the anonymous reviewers. In particular, one reviewer suggested extending the concept of residue to the qx+1 problem. Inspired by this suggestion, we define and study $Res_q\left( N \right) $, $res_q\left( N,N\prime \right) $, through which several interesting properties are obtained.


\medskip
\noindent MSC2020: 11B83



\begin{thebibliography}{00}

	\bibitem{Andaloro2000}
	Andaloro, P. (2000).
	On total stopping times under 3x + 1 iteration.
	\emph{Fibonacci Quart.} 38(1): 73--78.
	\href{https://doi.org/10.1080/00150517.2000.12428829}{doi.org/10.1080/00150517.2000.12428829}	

	\bibitem{Andrei2000}
	Andrei, S., Kudlek, M., Niculescu, R. S. (2000).
	Some results on the Collatz problem.
	\emph{Acta Inform.} 37(2): 145--160.
	\href{https://doi.org/10.1007/s002360000039}{doi.org/10.1007/s002360000039}
	
	\bibitem{Goodwin2015}
	Goodwin, J. R. (2015).
	The 3x+1 problem and integer representations. 
	\href{https://arxiv.org/abs/1504.03040}{arXiv:1504.03040}.

	\bibitem{Kontorovich2009}
	Kontorovich, A. V., Lagarias, J. C. (2009).
	Stochastic models for the 3x+1 and 5x+1 problems. 
	\href{https://arxiv.org/abs/0910.1944}{arXiv:0910.1944}.

	\bibitem{Ladue2018}
	Ladue, M. D. (2018).
	Clusters of integers with equal total stopping times in the 3x + 1 problem.
	\emph{Fibonacci Quart.} 56(2): 156--162.
	\href{https://doi.org/10.1080/00150517.2018.12427710}{doi.org/10.1080/00150517.2018.12427710}	
	
	\bibitem{Lagarias1985}
	Lagarias, J. C. (1985).
	The 3x+1 problem and its generalizations.
	\emph{Amer. Math. Monthly} 92(1): 3--23.
	\href{https://doi.org/10.1080/00029890.1985.11971528}{doi.org/10.1080/00029890.1985.11971528}	
	
	\bibitem{Lagarias2012}
	Lagarias, J. C. (2012).
	The 3x+1 problem: an annotated bibliography, II (2000-2009).
	\href{https://arxiv.org/abs/math/0608208}{arXiv:math/0608208}.
	
	\bibitem{Leavens1992}
	Leavens, G. T., Vermeulen, M. (1992).
	3x+1 search programs.
	\emph{Comput. Math. Appl.} 24(11): 79--99.
	\href{https://doi.org/10.1016/0898-1221(92)90034-F}{doi.org/10.1016/0898-1221(92)90034-F}
	
	\bibitem{Roosendaal2025}
	Roosendaal, E. (2025).
	On the 3x + 1 problem. 
	\href{https://www.ericr.nl/wondrous/}{www.ericr.nl/wondrous/}.
	
	\bibitem{Ruggiero2019}
	Ruggiero, R. (2019).
	The relationship between stopping time and number of odd terms in Collatz sequences. 
	\href{https://arxiv.org/abs/1911.01229}{arXiv:1911.01229}.
	
	\bibitem{Tao2022}
	Tao, T. (2022).
	Almost all orbits of the Collatz map attain almost bounded values.
	\emph{Forum Math. Pi} 10: e12 1--56.
	\href{https://doi.org/10.1017/fmp.2022.8}{doi.org/10.1017/fmp.2022.8}	

\end{thebibliography}
\end{document}